\documentclass[12pt]{amsart}
\usepackage{hyperref}
\usepackage{graphicx}

\usepackage[curve]{xypic}
\newbox\mybox
\def\overtag#1#2#3{\setbox\mybox\hbox{$#1$}\hbox to
  0pt{\vbox to 0pt{\vglue-#3\vglue-\ht\mybox\hbox to \wd\mybox
      {\hss$\ss#2$\hss}\vss}\hss}\box\mybox}
\def\undertag#1#2#3{\setbox\mybox\hbox{$#1$}\hbox to 0pt{\vbox to
    0pt{\vglue#3\vglue\ht\mybox\hbox to \wd\mybox
      {\hss$\ss#2$\hss}\vss}\hss}\box\mybox}
\def\lefttag#1#2#3{\hbox to 0pt{\vbox to 0pt{\vss\hbox to
      0pt{\hss$\ss#2$\hskip#3}\vss}}#1}
\def\righttag#1#2#3{\hbox to 0pt{\vbox to 0pt{\vss\hbox to
      0pt{\hskip#3$\ss#2$\hss}\vss}}#1}
\let\ss\scriptstyle

\def\Dot{\lower.2pc\hbox to 2pt{\hss$\bullet$\hss}}
\def\Circ{\lower.2pc\hbox to 2pt{\hss$\circ$\hss}}
\def\Vdots{\raise5pt\hbox{$\vdots$}}
\newcommand\lineto{\ar@{-}}
\newcommand\dashto{\ar@{--}}
\newcommand\dotto{\ar@{.}}

\newtheorem{theorem}{Theorem}[section]

\newtheorem{lemma}[theorem]{Lemma}
\theoremstyle{definition}
\newtheorem{definition}[theorem]{Definition}

\newtheorem{remark}[theorem]{Remark}
\newtheorem*{remark*}{Remark}
\newtheorem{example}[theorem]{Example}
\newtheorem{Method}{Method}

\newcommand{\PAL}{\operatorname{PPL}}
\newcommand{\PA}{\operatorname{PAL}}
\newcommand{\analytic}{algebraic}

\newcommand{\Brieskorn}{Brieskorn-Pham}
\def\case#1{\par\smallskip\noindent{\bf Case #1.}}

\begin{document}

\title{Principal analytic link theory in homology sphere links}
\author{A. N\'emethi}
\author{Walter D Neumann}
\author{A. Pichon}
\begin{abstract}
  For the link $M$ of a normal complex surface singularity $(X,0)$ we
  ask when a knot $K\subset M$ exists for which the answer to
  whether $K$ is the link of the zero set of some analytic germ
  $(X,0)\to (\mathbb C,0)$ affects the analytic structure on
  $(X,0)$. We show that if $M$ is an integral homology sphere then
  such a knot exists if and only if $M$ is not one of the Brieskorn
  homology spheres $M(2,3,5)$, $M(2,3,7)$, $M(2,3,11)$.
\end{abstract}
\maketitle

\section{Principal analytic link theory}

Let $M$ be a normal surface singularity link. In
particular, $M$ is a closed $3$--manifold which can be given by a
negative definite plumbing.

There may exist many different complex analytic structures on the
cone $C(M)$, i.e., many analytically different normal surface
singularities $(X,0)$ whose links $L_X$ are homeomorphic to $M$. Our aim
is to understand these different analytic structures from the point of
view of the ``{principal analytic link theory}'' on $M$.
  
A link or multilink $L = m_1 K_1\cup \ldots \cup m_r K_r \subset
M=L_X$ is \emph{\analytic} if $(M,L)$ is the link $(M,L)=(L_X,L_C)$ of
a germ pair $(X,C,0)$ consisting of a normal surface germ and a (not
necessarily reduced) complex curve through the singular point $0\in X$
(this was called ``analytic'' in \cite{NP}).  This is a topological
property: $L$ is \analytic{} if the $K_i$ are $\Bbb S^1$-fibres in a
negative definite plumbing decomposition of $M$ obtained by possibly
applying blow-ups to the minimal negative definite plumbing of $M$.

We say $L$ is \emph{principal analytic for $X$} if there exists 
a holomorphic germ $f\colon (X,0)
\rightarrow (\Bbb C,0)$ such that the pair $(M,L)$ is homeomorphic to
the link $(L_X,L_f)$ of the pair $(X,f^{-1}(0))$, taking account of 
multiplicities. 

We say that $L = m_1 K_1\cup \ldots \cup m_r K_r \subset M$ is
\emph{potentially principal} if there exists a normal surface germ $X$
with link $L_X=M$ for which $L$ is principal analytic.

According to (\cite{NP}, Theorem 2.1), the potential principality of
an algebraic
multilink $L \subset M$ is a 
topological property which is equivalent to any one of the following
\begin{itemize}
\item The multilink $(M,L)$ is fiberable;
\item  $[L] = 0$ in $H_1(M;\Bbb Z)$ (note that $[L]$ is always torsion);
\item  $I^{-1}\bf b$ is an integral vector, where $I$ is
the intersection matrix for the plumbing and $\bf b$ the vector whose
entry corresponding to a plumbing component is the sum of
multiplicities of components of $L$ that are fibres of that component.
\end{itemize}

When $M$ in the link of a rational singularity, then a potentially
principal multilink $(M,L)$ is principal analytic for every analytic
structure $(X,0)$ (\cite{Artin}).  The same conclusion holds when $M$ is
the link of a minimally elliptic singularity and $L$ is a knot
(\cite[Lemma p.~ 112]{Reid}).

In \cite{NP}, we gave several examples of surface singularity links
$M$ whose principal analytic link theory is sensitive to the analytic
structure in the following sense: for each analytic structure $(X,0)$
on $C(M)$, there exists a potentially principal knot in $M$ which is not
principal analytic for this structure. In fact, we gave examples of
pairs of potentially principal links, where the principality
of each obstructed the principality of the other.

The aim of this paper is to show that when $M$ is an integral homology
sphere ($\Bbb Z$HS) this behaviour is general, except in the rational
and minimally elliptic cases. Our technique consists of constructing a
set of principal analytic knots $K_1,\ldots,K_n$ which are not
compatible, {i.e.}, which cannot be realized by germs $f_i\colon (X,0)
\rightarrow (\Bbb C,0)$ from the same analytic structure
$(X,0)$.

\begin{example} Let $V(p,q,r) := \{(x_1, x_2,x_3) \in {\mathbb C}^3
  ~|~ x_1^p+x_2^q + x_3^{r} = 0\}$ with $p,q,r$ pairwise coprime. Its
  link $M=M(p,q,r)$ is a $\mathbb Z$--homology sphere with three
  singular fibres $K_1,K_2,K_3$ realized as principal analytic knots
  by $K_i=M \cap \{x_i = 0\}$.

  Let $K$ be the $(2,1)$-cable on $K_3\subset M(2,3,13)$. It is a
 potentially principal knot in $M= M(2,3,13)$.
  Let $(Z,p)$ be an analytic structure on the cone $C(M)$ such that
  $K_3$ is realized by a holomorphic function $f_3: (Z,p) \to (\mathbb
  C,0)$. Then $K$ is not realized by any $f\colon (Z,p) \to (\mathbb
  C,0)$ on $(Z,p)$ (\cite{NP}, 3,1).
\end{example}

Before stating more precisely the main result, let us generalize the
notion of principal analytic multilink, and say what we mean by
the principal analytic link theory of a surface singularity link $M$.
  
\begin{definition} A \emph{coloured multilink} in $M$ is the data of
  an algebraic multilink $L \subset M$ with a partition of its
  components: $L = L_1\coprod \ldots \coprod L_n$.
\end{definition}

\begin{definition}
  A coloured multilink $L = L_1\coprod \ldots \coprod L_n\subset M$ is
  \emph{principal analytic} for a normal surface singularity $(X,0)$
  with link $L_X=M$ if there exist analytic germs $f_i\colon (X,0)
  \rightarrow (\Bbb C,0)$, $i=1,\ldots,n$ such that
\begin{enumerate}
\item\label{it:1} the pair $(M,L)$ is homeomorphic to $(L_X,L_f)$, where $f =
  f_1\ldots f_r$;
\item\label{it:2} each $(M,L_i)$ is homeomorphic to
  $(L_X,L_{f_i})$ (note that this does not imply (1)---see Remark
  \ref{rk:isotopic knots}).
\end{enumerate}
We say $L$ is \emph{potentially principal} if it is principal analytic
for some analytic structure $(X,p)$. 
\end{definition}

Of course, the potential principality of each link $L_i$ is a
necessary condition for the potential principality of $L$. But it is
not sufficient when $n \geq 2$, as shown by the examples of
incompatible knots mentioned above: the coloured link $K_1\coprod
\ldots \coprod K_n$ is not potentially principal, but each component
is.  That is, the knots $K_1,\ldots,K_n$ can be realized by functions
$f_i\colon (X_i,0) \rightarrow (\Bbb C,0), i=1,\ldots,n$ defined on
some analytic structures $(X_i,0)$ on the cone $C(M)$, but the
$(X_i,0)$ cannot have the same analytical type.  So, although the
multilink $L = K_1\cup \ldots \cup K_n$ can be realized by a function
$f\colon (X,0) \rightarrow (\Bbb C,0)$ for some $(X,0)$, there is
no $(X,0)$ and $f$ such that $f$ splits into a product $f=f_1 \ldots
f_n$ with $f_i\colon (X,0) \rightarrow (\Bbb C,0)$ realizing the knot
$K_i$.

\vskip0,3cm
Given $M$, let us denote by $\PAL(M)$ the set of
potentially principal coloured multilinks $L$ in $M$;
we call $\PAL(M)$ the principal analytic theory on $M$. Given
a normal surface singularity $(X,0)$ with link $M$, we denote by
$\PA(X)\subset\PAL(M)$ the set of coloured links $L$ in
$M$ which are principally analytic for $(X,0)$. So
$\PAL(M)=\bigcup_{L_X\cong M}\PA(X)$.

The study of the principal analytic link theory on $M$ consists of the
two following natural questions:
\begin{enumerate}
\item Describe the set $\PAL(M)$;
\item describe the subsets $\PA(X)$ for $(X,0)$ realizing $M$.
\end{enumerate}

The unique rational singularity with $\Bbb Z$HS link is $(V(2,3,5),0)$
with link $M(2,3,5)$. There are only two $\Bbb Z$HS links which
belong to minimally elliptic singularities: $M(2,3,7)$ and
$M(2,3,11)$. Our main result, which is a first important step in this
program, is as follows:
 
\begin{theorem}
  \label{main} Let $M$ be a $\Bbb Z$HS singularity link
  which is not homeomorphic to $M(2,3,5)$, $M(2,3,7)$ or  $M(2,3,11)$. 
Then there exists an \analytic{} coloured link $L = K_1
  \coprod \ldots \coprod K_n$ which is
  not in $\PAL(M)$ and such that:
\begin{enumerate}
\item Each $K_i$ is a potentially principal knot
\item $\forall i \neq j$, $(M,K_i)$ is  not homeomorphic to $(M,K_j)$.
\end{enumerate}
\end{theorem}
(Of course, since $M$ is a $\Bbb Z$--homology sphere, the potential
principality of $K_i$ is
automatic.)

\section{Two constructions of non-PPL coloured links}

In this section, we present through examples two methods (Methods
\ref{method 1} and \ref{method 2}) to construct some coloured links
$L$ in a given $M$ such that $L \notin \PAL(M)$ but each $L_i$ is in
$\PAL(M)$. The first one, which was introduced in \cite{NP}, could be
used in any $M$, whereas the second is only available in a $\Bbb Z$HS.

\begin{Method}[Using the delta invariant of a reduced
  curve]\label{method 1}
Let $K$ be a fibred knot in $M$, and let $\Phi\colon M \setminus K
\rightarrow \Bbb S^1$ be an open-book fibration with binding $K$.  We
set $$\mu(K) = 1-\chi (\Phi^{-1}(t))\,,$$ where $ t \in \Bbb S^1$ and
where $\chi$ denotes the Euler characteristic. Notice that $\mu(K)$
does not depend on the choice of $\Phi$, and that it can be computed
from any plumbing graph of $(M,K)$ (or any splice diagram if $M$
is a $\Bbb Q$HS).
  
Let $ K_1\coprod \dots \coprod K_{n}$, $n \geq 2$ be a coloured link
whose components $K_i$ are potentially principal knots with
multiplicity $1$. For each $i=1,\ldots,n-1$, let $\Phi_i\colon M
\setminus K_i \rightarrow \Bbb S^1$ be a fibration of $K_i$.  We
consider the coloured multilink $L= K_1\coprod \ldots \coprod K_{n-1}$
and we define the semigroup ${\Gamma}(L,K_n)$ as the semigroup
generated by the degrees of the maps $\Phi_i$ on the knot
$K_n$. Notice that these degrees do not depend on the $\Phi_i$'s and
can be computed from any plumbing graph of $(M, K_1\coprod \ldots
\coprod K_{n}) $. We denote by $\delta({ L},K_n)$ the number of gaps
in ${\Gamma}({ L},K_n)$, i.e., the number of positive integers that
are not in $\Gamma(L,K_n)$.
\begin{lemma}\label{le:1}
  If  $ L \coprod K_n \in\PAL(M)$ then
$$ \mu(K_n) \leq 2   \delta({  L},K_n)\,.$$  
\end{lemma}
\begin{proof}Let $(X,0)$ be such $L \coprod K_n \in\PA(X)$ and let
  $f_j\colon(X,0)\rightarrow (\Bbb C,0)$ be a holomorphic germ with
  link $K_j$ for $j=1,\dots,n$. Then $\mu(K_n) = \mu(f_n)$, the Milnor
  number of $f_n$.  According to \cite{BG}, one has $\mu(f_n) = 2
  \delta(f_n)$, where $\delta(f_n)$ denotes the
  \emph{$\delta$-invariant} of the curve $f_n^{-1}(0)$. Recall that
  $\delta (f_n)$ counts the number of gaps in the semigroup
  $\Gamma(f_n)$ generated by the all the multiplicities of the
  holomorphic germs $g\colon (X,0)\rightarrow (\Bbb C,0)$ along the
  curve $f_n^{-1}(0)$.  Moreover, if $g$ is such a germ then this
  multiplicity is the degree of the Milnor fibration $g/|g|$
  restricted to the link of $f_n^{-1}(0)$. Since $K_1,\dots,K_{n-1}$
  can be realized by germs $f_1,\dots, f_{n-1}$, we have
  $\Gamma(L,K_n)\subset\Gamma(f_n)$, so $ \delta(f_n) \leq
  \delta(L,K_n)$.
\end{proof}
\end{Method}

\begin{example}[Non-$\PAL$ coloured link]
  Let $M$ be the link of the \Brieskorn{} singularity
  $z_1^3+z_2^4+z_3^5=0$ and and let $K_i$, $i=1,2,3$ be the end-knots
  in $M$ corresponding to $z_i=0$. Let us consider the $(2,5)-$cabling
  $K$ on the link $K_3$ of $z_3=0$. Its splice diagram is as follows.
$$
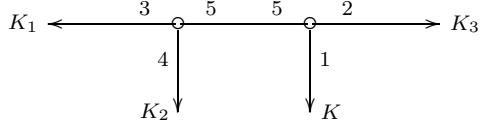
\xymatrix@R=15pt@C=24pt@M=0pt@W=0pt@H=0pt{
  {\ss K_1~}\ar@{<-}[r]&\lineto[r]^(.5){3}&\Circ
  \ar@{->}[dd]_(.4){4}
  \ar@{-}[rr]^(.25){5}^(.75){5}&&\Circ \ar@{->}[dd]^(.4)1
  \ar@{->}[rr]^(.25){2}&& {~\ss K_3}\\
  \\
  &&\lefttag{}{K_2}{3pt}&&\righttag{}{K}{4pt}
}$$ \vskip6pt\noindent
 The semigroup ${\Gamma}( K_1 \coprod K_2
  \coprod K,K_3)$, being generated by $3, 4$ and $5$, has two missing
  numbers, whereas $\mu(K_3)=(3-1)(4-1)=6>4$.

  Thus, by Lemma \ref{le:1}, the coloured link $L= K_1 \coprod K_2
  \coprod K_3 \coprod K$ is not realized on any $(X,0)$, i.e., $L
  \notin \PAL(M)$
\end{example}

\begin{Method}[Using the semigroup condition]\label{method 2}
Method 2 is based on the so-called End-Curve Theorem for $\Bbb
Z$HS links:

\begin{theorem} [\cite{NW2}, theorem 4.1] Let $(X,0)$ be a normal
  surface singularity with $\Bbb Z$HS link $M$. Let $\Delta$ be a
  splice diagram of $M$ such that $\Delta$ is the minimal splice
  diagram of the pair $(M,L)$, where $L$ denotes a link whose
  components are the end-knots of $\Delta$.

  Assume that for each of the end leaves of $\Delta$, there exits a
  function $z_i\colon (X,0) \rightarrow (\Bbb C,0)$ whose link is the
  corresponding end-knot. Then:
\begin{enumerate}
\item The graph $\Delta$ satisfies the semigroup condition;
\item $X$ is a complete intersection  of embedding dimension $\leq n$;
\item the functions $z_1,\ldots,z_n$ generate the maximal ideal of the
  local ring ${\mathcal O}_{(X,0)}$, and $X$ is a complete intersection of
  splice type with respect to these generators.\qed
\end{enumerate}
\end{theorem}

This result furnishes an alternative argument to prove that the $L$ of
the previous example does not belong to $\PAL(M(3,4,5))$. Indeed, if
$L \in \PA(X)$ for some analytic structure $(X,0)$ on $C(M)$ then
each leaf of the splice diagram of the figure is realized by a
function $(X,0) \rightarrow (\Bbb C,0)$, so, by the End-Curve Theorem,
$(X,0)$ is splice. But the semigroup condition is not realized at the
right hand node, as $5 \notin \langle 3,4\rangle$.

More generally, \emph{if a splice diagram does not
satisfy the semigroup condition, the coloured link consisting
of all end-knots for the diagram  is not in $\PAL(M)$ by the
End-Curve Theorem, so Theorem  \ref{main} is proved in
this case}.  

Not all splice diagrams for $\mathbb Z$HS singularity links satisfy
the semigroup conditions. Nevertheless, Method 2 gives a 
short proof of a weak version of Theorem \ref{main}:
\end{Method}

\begin{theorem}[Weak Version of Theorem \ref{main}]\label{weakmain} 
  Let $M$ be a $\Bbb Z$-homology sphere which is the link of a normal
  surface singularity.  If $M$ is not homeomorphic to $M(2,3,5)$ then
  there exists a coloured link $L = K_1\coprod \ldots \coprod
  K_n\notin \PAL(M)$, consisting of knots $K_i\in\PAL(M)$.
\end{theorem}
  
\begin{proof}
Let $G$ be the minimal
  splice diagram of $M$.  Let us consider an end-node of $G$ as in the
  figure below, with $a_1< \ldots <a_n$.  Let $ K_1, \ldots, K_n$ be
  end-knots as marked, and $ K_{n+1},
  \ldots, K_r$ the end-knots corresponding to the remaining leaves
  (lying in the portion $G'$). 
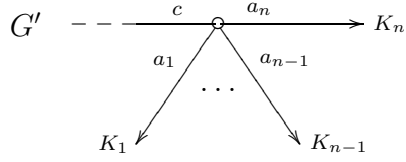
$$ 
\xymatrix@R=20pt@C=24pt@M=0pt@W=0pt@H=0pt{
  {G'\quad}\dashto[r]&\lineto[r]^(.5){c}&\Circ
  \ar@{->}[ddl]_(.4){a_1}\ar@{->}[ddr]^(.4){a_{n-1}}
  \ar@{->}[rr]^(.25){a_n}&&{~\ss K_n}\\
  &&\dots&\\
  &\lefttag{}{K_1}{3pt}&&\righttag{}{K_{n-1}}{4pt} }$$ 
We can assume $c>1$, since we have already proved the result if the
semigroup condition fails.
Denote
$$A=\prod_{i=1}^{n-1}a_i\,;\quad A_j=A/a_j,~j=1,\dots,n-1\,.$$
Set $\alpha = c
A_{n-1}$. Assume we can choose $d \in \{\alpha +1, \alpha
+2\}$ such that $d \notin \langle\alpha, A\rangle$.
Replace $K_n$ by two parallel $(1,d)$ cablings on $K_n$ as shown.
$$ 
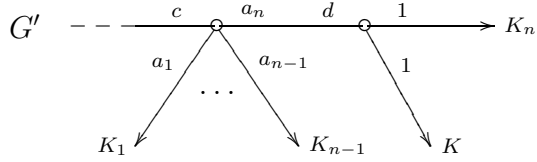
\xymatrix@R=20pt@C=24pt@M=0pt@W=0pt@H=0pt{
  {G'\quad}\dashto[r]&\lineto[r]^(.5){c}&\Circ
  \ar@{->}[ddl]_(.4){a_1}\ar@{->}[ddr]^(.4){a_{n-1}}
  \ar@{-}[rr]^(.25){a_n}^(.75){d}&&\Circ \ar@{->}[ddr]^(.4)1
  \ar@{->}[rr]^(.25){1}&& {~\ss K_n}\\
  &&\dots&&&&\\
  &\lefttag{}{K_1}{3pt}&&\righttag{}{K_{n-1}}{4pt}&&\righttag{}{K}{4pt}
}$$ As $a_n d > \alpha a_{n-1}$, the splice diagram of $(M,K_n \cup
K)$ satisfies the edge determinant condition.  However, the semigroup
condition is not realized at the right hand node as
$d<cA_j$ for $j<{n-1}$ and $d \notin
\langle\alpha, A\rangle$.  Thus, by the End-Curve
Theorem, the fibered coloured link $L = K_1 \coprod \ldots \coprod K_r
\coprod K$ does not belong to $\PAL(M)$. This completes the proof,
assuming that $d$ exists.

If $d$ above does not exist then $n=2$ and $(c,a_1)=(2,3)$ or $(c,2)$
with $c$ odd. In the latter case, if $c\ge7$ then $d=c-2$ satisfies
$a_2d>a_1c$ but fails the semigroup condition, so we may assume $c=3$
or $5$.

Moreover, if $G'$ does not just consist of a single vertex, then the
smallest multiple of $a_1$ that can contribute to the semigroup is
$2a_1$ so we see that the semigroup condition still fails with $d=3$
if $(c,a_1)=(2,3)$ and with $d=c+2$ if $(c,a_1)=(3,2)$ or $(5,2)$.
Finally, for the one-node diagram $G$ with weights $2,5,a_2$ we can
use $d=3$ while for $2,3,a_2$ we can use $d=1$ (recall 
$a_2\ge 7$ in both cases
since we ruled out $M(2,3,5)$).
$$
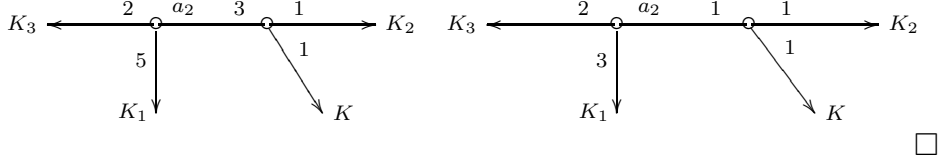
\xymatrix@R=15pt@C=20pt@M=0pt@W=0pt@H=0pt{
  {\ss K_3~}\ar@{<-}[r]&\lineto[r]^(.5){2}&\Circ
  \ar@{->}[dd]_(.4){5}
  \ar@{-}[rr]^(.25){a_2}^(.75){3}&&\Circ \ar@{->}[ddr]^(.4)1
  \ar@{->}[rr]^(.25){1}&& {~\ss K_2}\\
  \\
  &&\lefttag{}{K_1}{3pt}&&&\righttag{}{K}{4pt}
}\quad
\xymatrix@R=15pt@C=24pt@M=0pt@W=0pt@H=0pt{
  {\ss K_3~}\ar@{<-}[r]&\lineto[r]^(.5){2}&\Circ
  \ar@{->}[dd]_(.4){3}
  \ar@{-}[rr]^(.25){a_2}^(.75){1}&&\Circ \ar@{->}[ddr]^(.4)1
  \ar@{->}[rr]^(.25){1}&& {~\ss K_2}\\
  \\
  &&\lefttag{}{K_1}{3pt}&&&\righttag{}{K}{4pt}
}$$ 
\end{proof}

\begin{remark}\label{rk:isotopic knots}
  The knots $ K_n$ and $K$ are isotopic.  That is why Theorem
  \ref{weakmain} is a weak version of Theorem \ref{main}. The
  positions of isotopic knots with respect to each other can make a
  big difference. In the following two splice diagrams (with $a\ge 7$)
  the knots $K_2$, $K'_2$, $K''_2$ and $K'''_2$ are mutually isotopic.
$$
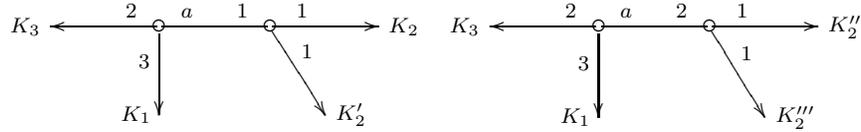
\xymatrix@R=15pt@C=20pt@M=0pt@W=0pt@H=0pt{
  {\ss K_3~}\ar@{<-}[r]&\lineto[r]^(.5){2}&\Circ \ar@{->}[dd]_(.4){3}
  \ar@{-}[rr]^(.25){a}^(.75){1}&&\Circ \ar@{->}[ddr]^(.4)1
  \ar@{->}[rr]^(.25){1}&& {~\ss K_2}\\
  \\
  &&\lefttag{}{K_1}{3pt}&&&\righttag{}{K'_2}{4pt} }\quad
\xymatrix@R=15pt@C=20pt@M=0pt@W=0pt@H=0pt{ {\ss
    K_3~}\ar@{<-}[r]&\lineto[r]^(.5){2}&\Circ \ar@{->}[dd]_(.4){3}
  \ar@{-}[rr]^(.25){a}^(.75){2}&&\Circ \ar@{->}[ddr]^(.4)1
  \ar@{->}[rr]^(.25){1}&& {~\ss K''_2}\\
  \\
  &&\lefttag{}{K_1}{3pt}&&&\righttag{}{K'''_2}{4pt} }$$ As just
proved, the left colored link is not in $\PAL(M)$. But for $a=7$ or $11$
the right one is in $\PA(X)$ for every analytic structure $X$ on the cone
$C(M)$.  The reason is that $X$ is minimally elliptic and $K''_2$ and
$K'''_2$ are realized by generic hyperplane sections.
\end{remark}

\section{Proof of Theorem \ref{main} }
     
\subsection{The case of \Brieskorn{} links}
Let $M$ be the link of the \Brieskorn{} singularity
$z_1^p+z_2^q+z_3^r=0$ where $p<q<r$ are pairwise coprime integers.
Let $K_i$, $i=1,2,3$ be the end-knots in $M$ corresponding to $z_i=0$.
       
If there exists $s \in \Bbb N$ such that
$$rs > 2pq \text{ and } s\notin \langle p,q\rangle \hskip1cm (\ast),$$ 
then, using the semigroup condition (Method \ref{method 2}), one
obtains that the four-coloured link $L= K_1 \coprod K_2 \coprod K_3
\coprod K $ does not belong to $\PAL(M)$, where $K$ denotes the
$(2,s)$--cabling on $K_3$:
$$
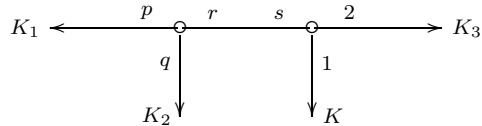
\xymatrix@R=15pt@C=24pt@M=0pt@W=0pt@H=0pt{
  {\ss K_1~}\ar@{<-}[r]&\lineto[r]^(.5){p}&\Circ
  \ar@{->}[dd]_(.4){q}
  \ar@{-}[rr]^(.25){r}^(.75){s}&&\Circ \ar@{->}[dd]^(.4)1
  \ar@{->}[rr]^(.25){2}&& {~\ss K_3}\\
  \\
  &&\lefttag{}{K_2}{3pt}&&\righttag{}{K}{4pt}
}$$
First, assume that $p >2$. The integers in the semigroup $\langle
p,q\rangle$ which are $\leq 2p+1$ belong to $\{p,q,2p,p+q \}$.  Then
if $q \notin \{ p+1, 2p+1\}$, $s = 2p+1$ satisfies condition $
(\ast)$.

If $q = p+1$, the integers in the semigroup $\langle p,q\rangle$
which are $\leq 2p+3$ belong to: $\{ p,p+1,2p,2p+1,2p+2, 3p\}$.  Then, if
$p\neq 3$, $s = 2p+3$ satisfies condition $ (\ast)$ as $2p+2
<2p+3<3p$, and if $(p,q)=(3,4)$, then one can choose $s=5$ as in the
example of section 2.

If $q=2p+1$ and $p >3$, then $s = 2p+3$ satisfies condition $ (\ast)$,
and if $(p,q)=(3,7)$, one can take $s=11$.

It remains to deal with the case that $p=2$.
 If $q>5$, then $s=5$ satisfies condition $(\ast)$. 
 If $(p,q)=(2,5)$, then $r \geq 7$ and $s=3$ satisfies $(\ast)$. If
 $(p,q)=(2,3)$, then $r \geq 13$ (as we avoid the rational case $r=5$
 and  minimally elliptic cases $r=7,11$) and $s=1$ satisfies $(*)$.

This completes the proof for \Brieskorn{} links.

\subsection{The case of a Seifert link}
 
Assume that the $3$-manifold $M$ is Seifert, or equivalently that the
minimal splice diagram $G$ of $M$ has a single node.  Let $n$ be the
number of incident leaves. We assume that $n \geq 4$, as the case $n
\leq 3$, which corresponds to the \Brieskorn{} case, has already been
treated.  Let $a_1,\ldots,a_n$ be their weights, indexed in such a way
that $a_1<\ldots<a_n$, and let $K_1,\ldots,K_n $ be corresponding
end-knots. We set:
  $$A=\prod_{i=1}^{n-1}a_i;\quad A_j=A/a_j, ~j=1,\dots,n-1\,.$$
  
We argue as in the \Brieskorn{} case: If there exists $s \in \Bbb N$
such that
$$a_ns > 2 A\quad   \text{and}\quad s\notin \langle
A_{1},\dots,A_{n-1} \rangle \hskip1cm (\ast_2),
$$ 
then the $(n+1)$-coloured link $L= K_1 \coprod \ldots \coprod K_n
\coprod K $ does not belong to $\PAL(M)$, where $K$ denotes the
$(2,s)$- cabling on $K_n$ (see figure).
$$
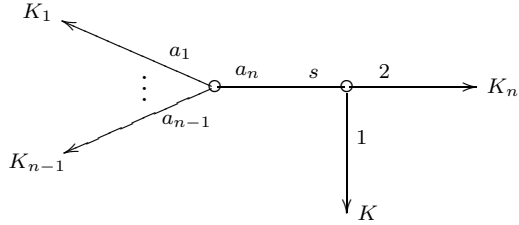
\xymatrix@R=15pt@C=24pt@M=0pt@W=0pt@H=0pt{
  {\ss K_1~}\ar@{<-}[drr]^(.75){a_1}\\
&\Vdots&\Circ
  \ar@{-}[rr]^(.25){a_n}^(.75){s}&&\Circ \ar@{->}[dd]^(.4)1
  \ar@{->}[rr]^(.25){2}&& {~\ss K_n}\\
  {\ss K_{n-1}~}\ar@{<-}[urr]_(.75){a_{n-1}}
  \\
  &&&&\righttag{}{K}{4pt}
}$$
We will show that $s=2A_{n-1}+1$ satisfies
$(\ast_2)$. 

First notice that $s$ satisfies the inequality of condition
$(\ast_2)$. Moreover, as $n \geq 4$, we have
 $$6 \leq A_{n-1} < A_{n-2} < \ldots<A_1\,,$$
 and for each $i\in \{2,\ldots,n-1\}$, one has $A_{i-1} - A_{i} \ge
 3$. Thus the integers in the semigroup $ \langle A_{j} \ ; \ j
 =1,\ldots,n-1 \rangle $ which are $\leq 2A_{n-1} +1$ must be among $
 A_{n-1}, \ldots, A_1$ and $2A_{n-1}$, and hence divisible by one of
 the $a_i$'s with $i<n-1$. So $2A_{n-1}+1$ cannot be in this
 semigroup, so it satisfies $(\ast_2)$.

 \begin{remark*}
   One can also prove Theorem \ref{main} in the case of a Seifert link
   using only Method \ref{method 1}, with much more complicated
   cases. The advantage of such a proof is that Method 1 can be used
   in any $3$-dimensional manifold with the same underlying splice
   diagram, even those having genus at the nodes and leaves, whereas
   the second method is specific to $\Bbb Z$HS's. So one obtains an
   extension of Theorem \ref{main} to a larger family of normal
   singularity links.

 The idea of the proof in the case of a \Brieskorn{} link $M =\Bbb
 S^3_{\epsilon } \cap \{z_1^p+z_2^q + z_3^r=0\}$ is the following:
 using again the notation $K_1$, $K_2$, $K_3$ for the end-links, the
 generic cases are treated by performing a string of cablings on $K_3$
 giving rise to $x+1$ knots $K'_0, \ldots,K'_{x}$ as in the figure
 below, where $x=2a+1$.
$$
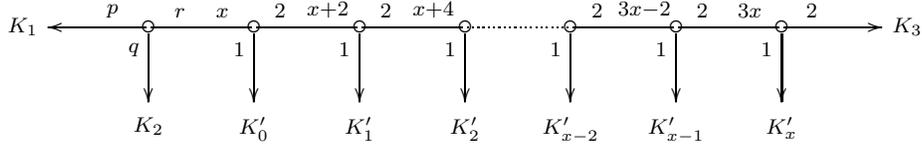
\xymatrix@R=25pt@C=36pt@M=0pt@W=0pt@H=0pt{
{\ss K_1~}\ar@{<-}[r]^(.7){p}&\Circ
\ar@{->}[d]_(.3)q\lineto[r]^(.3){r}^(.7){x}&
\Circ\ar@{->}[d]_(.3)1\lineto[r]^(.25)2^(.7){x+2}&
\Circ\ar@{->}[d]_(.3)1\lineto[r]^(.25)2^(.7){x+4}&
\Circ\ar@{->}[d]_(.3)1\ar@{.}[r]&
\Circ\ar@{->}[d]_(.3)1\lineto[r]^(.25)2^(.7){3x-2}&
\Circ\ar@{->}[d]_(.3)1\lineto[r]^(.25)2^(.7){3x}&
\Circ\ar@{->}[d]_(.3)1\ar@{->}[r]^(.25)2&{\ss~K_3}\\
&\undertag{~}{K_2}{6pt}&\undertag{~}{K'_0}{6pt}&
\undertag{~}{K'_1}{6pt}&\undertag{~}{K'_2}{6pt}&
\undertag{~}{K'_{x-2}}{6pt}&\undertag{~}{K'_{x-1}}{6pt}
&\undertag{~}{K'_x}{6pt}\\
}$$\vskip 18pt\noindent
Then we show that, except for a finite numbers of particular cases, we
have the inequality:
$$ \mu(K_3) > \delta({ \mathcal L},K_3),$$
where ${\mathcal L} = \{K_1,K_2, K'_0, \ldots, K'_{x}\}$.
So, by Lemma 1, the (x+4)-coloured link $$K_1 \coprod K_2
\coprod K_3 \coprod K'_0 \coprod \ldots \coprod K'_x$$ does not belong
to $\PAL(M)$. 

The particular cases which cannot be treated by these cablings, but
which can be solved independently by hand are the following:
$(p,q) \in \{ (2,3)$, $(2,5)$, $(2,7)$, $(3,4)$, $(3,5)$, $(4,5)$,
$(4,7)$, $(6,7)\}$.  The details are left to the reader. The proof in
the general Seifert case is a generalization of this one. It is likely
that a similar proof exists for a general normal surface
singularity with $\Bbb Z$HS link, but we have not attempted it.
 \end{remark*}

 \subsection{Proof of Theorem \ref{main} in the non-Seifert case}
\label{subsec:non-seifert}
 
Let us assume that the splice diagram $G$ has at least two
nodes. Choose an \emph{end-node} $(\nu)$ of the splice diagram $G$,
that is, a node which is an end-vertex of the diagram obtained by
removing all leaves from $G$ (so it has at most
a single incident edge which is not a leaf).

\subsubsection{\bf First case:  $(\nu)$ has $4$ or more incident edges} 

Let $n+1$ be the number of incident edges of $\nu$.  Denote by
$a_1<\ldots<a_n$ the weights on the adjacent leaves and by $r$ that on
the remaining adjacent edge.  (We can assume $r>1$ if we want, since
otherwise the diagram fails the semigroup condition, and we have
already proved this case.)

Let $K_1,\ldots,K_n,K_{n+1},\ldots,K_m$ be
end-knots corresponding to all the leaves of $G$, the knots
$K_1,\ldots,K_n$ corresponding to the leaves adjacent to $\nu$. We set
$$A:=\prod_{i< n} a_i;\quad A_j:=\prod_{i \neq j, n} a_i, \text{ for }j=1,\dots,n-1\,.$$

  We argue as before: If there exists $s \in \Bbb N$ such that
$$a_ns > 2r A_{n-1} a_{n-1} \hbox{ and } s\notin 
\langle A, rA_1,rA_2,\ldots,rA_{n-1}\rangle \hskip1cm
(\ast_3),$$ then the $(m+1)$-coloured link $L= K_1 \coprod \ldots
\coprod K_m \coprod K $ does not belong to $\PAL(M)$, where $K$
denotes the $(2,s)$--cabling on $K_n$.
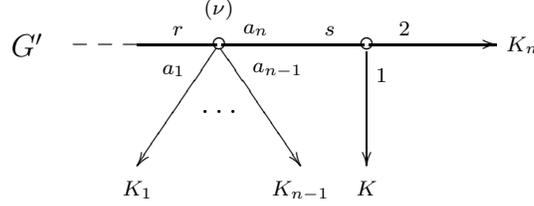
 $$ 
 \xymatrix@R=20pt@C=24pt@M=0pt@W=0pt@H=0pt{
   {G'\quad}\dashto[r]&\lineto[r]^(.5){r}&\overtag{\Circ}{(\nu)}{15pt}
   \ar@{->}[ddl]_(.3){a_1}\ar@{->}[ddr]^(.3){a_{n-1}}
   \lineto[rr]^(.25){a_n}^(.75){s}&&\Circ \ar@{->}[rr]^(.25)2
   \ar@{->}[dd]^(.25){1}&&{~\ss K_n}\\
   &&\dots&&&&\\
   &\undertag{}{K_1}{6pt}&&\undertag{}{K_{n-1}}{6pt}&\undertag{}{K}{6pt}}
$$\vskip 12pt\noindent
We will show that $s=2rA_{n-1}+1$ satisfies $(\ast_3)$. 

As
$a_n>a_{n-1}$, $s=2rA_{n-1}+1$ satisfies the inequality of condition
$(\ast_3)$.
We again have
 $$6\le A_{n-1} < A_{n-2} < \ldots<A_1,$$
and  $A_{i-1} - A_i \ge 3$ for $2\le i \le n-1 $.

Since $rA_i>rA_{n-1}+3r$ for $i<n-1$, the integers in the semigroup
$\Gamma = \langle A, rA_1,rA_2,\ldots,rA_{n-1}\rangle $ which are
$\leq 2rA_{n-1} +1$ must be among
  $$ 2rA_{n-1},~ nA,~ rA_i+nA,\quad n=0,1,2,\dots $$
These are all divisible by some $a_i$ with $i<n-1$ and $2rA_{n-1} +1 $
is not, so it cannot be in the semigroup.

\subsubsection{\bf Second case:  $(\nu)$ has $3$ incident edges} 
We denote by $p<q$ the weights of the two adjacent leaves and by $r$
that of the remaining edge.
 
If $r <q$, then we can use the same argument as in the \Brieskorn{}
case. (In fact the argument is simplified by the fact that the semigroup
in the argument is now smaller than $\langle r,p\rangle$ since it is
contained in the semigroup generated by $r,2p,3p,\dots$. So some cases
of the earlier argument are not needed and a $(2,2t+1)$--cabling at the
$q$--weighted leaf works with $t = \min (p,r)$ unless $r=2p+1$, in which
case $(2,2p+3)$--cabling works.)
 
Let us now assume $p<q<r$. Let $K_1$ and $K_2$ be end-knots
corresponding to the leaves weighted by $p$ and $q$ respectively, and
let $K_3,\ldots,K_m$ be end-knots corresponding to the other leaves of
the splice diagram $G$.

For each $i=3,\ldots,m$, let $\alpha'_{i}$ be the product of the
weights adjacent to the path joining the leaf $K_2$ to the leaf
$K_i$. As $p$ divides $\alpha'_i$, we set $\alpha'_i =p \alpha_i $,
where $\alpha_i \geq 2$. If there exists $s \in \Bbb N$ such that
$$qs > 2pr \hbox{ and } s\notin \langle r,\alpha_3p, \ldots,p\alpha_m 
\rangle \hskip1cm (\ast_4),
$$ 
then the $(m+1)$--coloured link $L= K_1 \coprod K_2 \coprod K_3 \coprod
\ldots \coprod K_m \coprod K $ does not belong to $\PAL(M)$, where $K$
is the $(2,s)$--cabling on $K_2$:
 $$ 
 
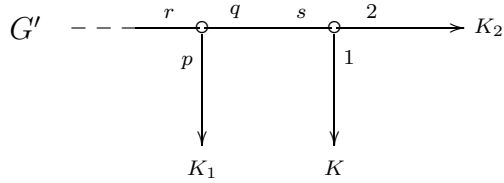
\xymatrix@R=20pt@C=24pt@M=0pt@W=0pt@H=0pt{
   {G'\quad}\dashto[r]&\lineto[r]^(.5){r}&{\Circ}
   \ar@{->}[dd]_(.3){p}
   \lineto[rr]^(.25){q}^(.75){s}&&\Circ \ar@{->}[rr]^(.25)2
   \ar@{->}[dd]^(.25){1}&&{~\ss K_2}\\
   &&&&&&\\
   &&\undertag{}{K_1}{6pt}&&\undertag{}{K}{6pt}}
$$\vskip 12pt\noindent

As $p<q$, $s=2r +1$ and $s=2r+3$ both satisfy the inequality of
condition $(\ast_4)$.
Assume that $2r +1 $ and $2r +3 $ both belong to the semigroup
$\langle r,\alpha_3p, \ldots,p\alpha_m \rangle$. There then exist
$\kappa,\gamma \in \langle \alpha_3,\ldots,\alpha_m\rangle$ such that
\begin{align*}
   2r+1 = \kappa p\quad&\text{or}\quad 2r+1 = r + \kappa p,\quad\text{and}\\  
   2r+3 = \gamma p \quad &\text{or}\quad 2r+3 = r +  \gamma p \,.
\end{align*}
(The possibilities $2r+3=2r+\gamma p$ and $2r+3=3r$ are ruled out by
the facts $\gamma>1$ and $r\ge 5$.)  
Thus there are four possible cases:
\case1 $ 2r+1 = \kappa p$ and $2r+3 = \gamma p$. Then $2=(\gamma -
  \kappa)p$, so $p=2$, so $2r+1=2\kappa$. Contradiction.
\case2 $2r+1 = r + \kappa p $ and $2r+3 = \gamma p $, 
which leads to $1=(\gamma - 2\kappa)p$. Contradiction.
\case3 $2r+1 = \kappa p $ and $2r+3 = r+ \gamma p $, which leads to
  $5=(2 \gamma - \kappa)p$. Therefore $p=5$. As $2\leq p<q<r$,
  $s=2r-1$ satisfies the inequality of condition $(\ast_4)$. Assume
  that $2r-1$ and $2r+1$ both belong to $\langle r,5 \alpha_3,
  \ldots,5\alpha_m \rangle$. There then exist $\lambda,\delta \in
  \langle \alpha_3,\ldots,\alpha_m\rangle$ such that
  \begin{align*}
2r+1 = 5\lambda\quad&\text{or}\quad 2r+1 = r + 5\lambda,\quad\text{and}\\ 
2r-1 = 5\delta \quad&\text{or}\quad 2r-1 = r + 5\delta\,.  
  \end{align*}
This leads to four possible cases:
 \begin{enumerate}
\item   $ 2r+1 = 5\lambda$ and  $2r-1 = 5\delta $. 
Then $2=5(\lambda-\delta)$. Contradiction.
\item $2r+1 = r + 5\lambda $ and $2r-1 =5 \delta $, which leads to
  $3=5(2\lambda - \delta )$.  Contradiction.
\item $2r+1 = r + 5\lambda$ and $2r-1 = r + 5 \delta $, which leads to
  $2=5(\lambda-\delta)$.  Contradiction
\item $2r+1 =  5\lambda $ and $2r-1 = r +5\delta  $, which leads to $3=5( \lambda- 2 \delta)$.  Contradiction
\end{enumerate}
\case4 $2r+1 = r + \kappa p$ and $2r+3 = r + \gamma p$. This leads to
  $2=(\gamma - \kappa)p$, so $p=2$.
If $q\ge 5$ then  $s=r+2$ satisfies condition $(\ast_4)$. So we may
assume $q=3$. Then $r\ge 5$ and for $r=5$ we can take $s=7$, so we may
assume $r\ge 7$.  

Since we have dealt with all other possibilities, we can assume now
that \emph{every} end-node of $G$ has this form, i.e., it is valence $3$
with two leaves with weights $2$ and $3$ and the third edge having weight
$\ge 7$.

Let $(\nu)$ now be an end-node of the graph obtained by deleting the
$2$-- and $3$--weighted leaves at the end-nodes of $G$.  So the
picture is as follows:
 $$ 
 
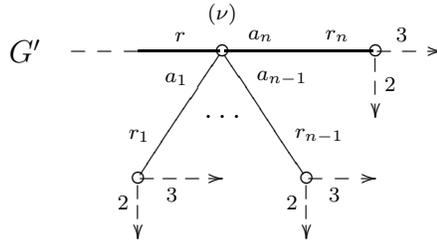
\xymatrix@R=20pt@C=24pt@M=0pt@W=0pt@H=0pt{
   {G'\quad}\dashto[r]&\lineto[r]^(.5){r}&\overtag{\Circ}{(\nu)}{15pt}
   \lineto[ddl]_(.3){a_1}_(.75){r_1}\lineto[ddr]^(.3){a_{n-1}}^(.75){r_{n-1}}
   \lineto[rr]^(.25){a_n}^(.75){r_n}&&\Circ \ar@{-->}[d]^(.5)2
   \ar@{-->}[r]^(.4){3}&\\
   &&\dots&&&&\\
   &\Circ\ar@{-->}[d]_(.4)2\ar@{-->}[r]_(.4)3&&\Circ\ar@{-->}[d]_(.4)2\ar@{-->}[r]_(.4)3&\\
   &&&}$$ The dashed pairs of $(2,3)$--weighted leaves may or may not
 exist (but at least one pair must exist; note that the weight $r_i$
 is $\ge7$ but is omitted if the $(2,3)$--pair of leaves does not
 exist at that vertex).  We assume $a_1<\dots<a_n$, and denote, as
 usual, $A=a_1\dots a_{n-1}$ and $A_j=A/a_j$ for $j=1,\dots,n-1$.

 We first consider the case that there is no $(2,3)$--pair at
 the end of the $a_n$--weighted edge:
 $$ 
 
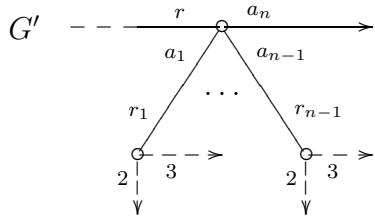
\xymatrix@R=20pt@C=24pt@M=0pt@W=0pt@H=0pt{
   {G'\quad}\dashto[r]&\lineto[r]^(.5){r}&\Circ
   \lineto[ddl]_(.3){a_1}_(.75){r_1}\lineto[ddr]^(.3){a_{n-1}}^(.75){r_{n-1}}
   \ar@{->}[rr]^(.25){a_n}&&\\
   &&\dots&&&&\\
   &\Circ\ar@{-->}[d]_(.4)2\ar@{-->}[r]_(.4)3&&\Circ\ar@{-->}[d]_(.4)2\ar@{-->}[r]_(.4)3&\\
   &&&}$$ Then the same arguments as before reduce us to the case
that $n=2$, $a_1=2$, $a_2=3$, and then the following cabling does not
 satisfy the semigroup condition and thus resolves this case:
 $$ 
 
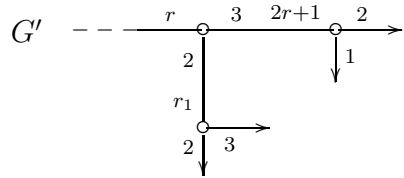
\xymatrix@R=15pt@C=24pt@M=0pt@W=0pt@H=0pt{
   {G'\quad}\dashto[r]&\lineto[r]^(.5){r}&\Circ
   \lineto[dd]_(.3){2}_(.75){r_1}
   \ar@{-}[rr]^(.25){3}^(.7){2r+1}&&
\Circ\ar@{->}[d]^(.5)1\ar@{->}[r]^(.4)2&\\
   &&&&&&\\
   &&\Circ\ar@{->}[d]_(.4)2\ar@{->}[r]_(.4)3&&&\\
   &&&}$$
Indeed, the relevant semigroup is generated by $2r$, $3r$, and a
subset of $2\mathbb N$, so it does not contain $2r+1$.

We may now assume there is a $(2,3)$--pair at the $a_n$--weighted
edge. If there is an integer $k$ with $\frac{r_n}6>k>\frac{rA}{a_n}$
then the following internal cabling gives an admissible splice diagram
for $M$, which fails the semigroup condition since $1\notin\langle2,3\rangle$:
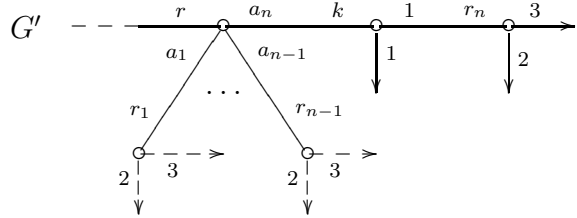
 $$ 
 \xymatrix@R=20pt@C=24pt@M=0pt@W=0pt@H=0pt{
   {G'\quad}\dashto[r]&\lineto[r]^(.5){r}&\Circ
   \lineto[ddl]_(.3){a_1}_(.75){r_1}\lineto[ddr]^(.3){a_{n-1}}^(.75){r_{n-1}}
   \lineto[rr]^(.25){a_n}^(.75){k}&&\Circ\ar@{->}[d]^(.4)1
   \lineto[rr]^(.25){1}^(.75){r_n}&&
   \Circ\ar@{->}[d]^(.5)2\ar@{->}[r]^(.4){3}&\\
   &&\dots&&&&&&\\
   &\Circ\ar@{-->}[d]_(.4)2\ar@{-->}[r]_(.4)3&&\Circ\ar@{-->}[d]_(.4)2\ar@{-->}[r]_(.4)3&\\
   &&&}$$ So we may assume that there is no integer $k$ with
 $\frac{r_n}6>k>\frac{rA}{a_n}$ (but $\frac{r_n}6>\frac{rA}{a_n}$).  In particular,  $rA_{n-1}\ge \left\lceil
 \frac{r_n}6\right\rceil$, since $rA_{n-1}$ is an integer larger
 than $\frac{rA}{a_n}$.  We thus have:
$$ rA_1>\dots>rA_{n-1}\ge \left\lceil
 \frac{r_n}6\right\rceil\,.$$
Consider the cabling:
 $$ 
 
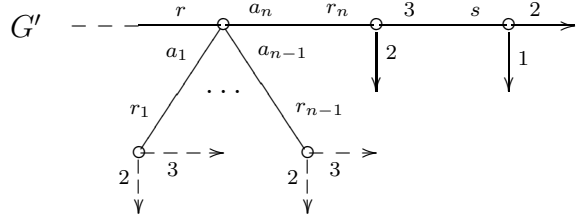
\xymatrix@R=20pt@C=24pt@M=0pt@W=0pt@H=0pt{
   {G'\quad}\dashto[r]&\lineto[r]^(.5){r}&\Circ
   \lineto[ddl]_(.3){a_1}_(.75){r_1}\lineto[ddr]^(.3){a_{n-1}}^(.75){r_{n-1}}
   \lineto[rr]^(.25){a_n}^(.75){r_n}&&\Circ\ar@{->}[d]^(.4)2
   \lineto[rr]^(.25){3}^(.75){s}&&
   \Circ\ar@{->}[d]^(.5)1\ar@{->}[r]^(.4){2}&\\
   &&\dots&&&&&&\\
   &\Circ\ar@{-->}[d]_(.4)2\ar@{-->}[r]_(.4)3&&\Circ\ar@{-->}[d]_(.4)2\ar@{-->}[r]_(.4)3&\\
   &&&}$$ in which we need $4r_n<3s$ to have an admissible
 diagram.  Let $s=r_n+2x$ with $x=\left\lceil \frac{r_n}6\right\rceil+\epsilon$, $\epsilon\in\{0,1\}$. Then
 $3s=3r_n+6x> 4r_n$, as desired. Moreover, since $r_n>6 $, we have
 $x<r_n$.

 Since $s=r_n+2x$ is odd and $x<r_n$, if $s$ satisfies the semigroup
 condition then $x$ must be in the semigroup generated by
 $rA_1,\dots,rA_{n-1}$ and $A$.   Thus if both $x=\left\lceil
 \frac{r_n}6\right\rceil$ and $x=\left\lceil \frac{r_n}6\right\rceil+1$ are in this
 semigroup there are just two possibilities:
 \begin{align*}
   \left\lceil \frac{r_n}6\right\rceil=rA_{n-1} \quad&\text{and}\quad A\text{
     divides }\left\lceil \frac{r_n}6\right\rceil+1\,,\quad\text{or}\\ 
   \left\lceil \frac{r_n}6\right\rceil+1=rA_{n-1} \quad&\text{and}\quad A\text{
     divides }\left\lceil \frac{r_n}6\right\rceil\,.
 \end{align*}
 Both cases imply $n=2$, since otherwise $a_1$ is a divisor of both
 $\left\lceil \frac{r_n}6\right\rceil$ and $\left\lceil
 \frac{r_n}6\right\rceil+1$. Moreover, they imply there is no
$(2,3)$--pair at the end of the $a_1$--edge, since if there were then
$rA_{n-1}=r$ would not be available in the semigroup. 

Now suppose that $x=\left\lceil \frac{r_n}6\right\rceil+2$ is also in
the semigroup. In the first case $x=r+2$; the only possibilities are
$r=2$ (so $a_1=3$ since $A=a_1$ divides $r+1$) or $a_1=2$. In either
case we need $a_1$ to be in the semigroup, so the $r$--weighted edge
is a leaf, so node $(\nu)$ is an end-node, so the weights of its
leaves are $2$ and $3$ and $a_2\ge 7$.  The second case similarly
implies that node $(\nu)$ must be an end-node so its pair of leaves is
$(2,3)$--weighted and the third weight $a_2$ is $\ge7$.

Since $a_2\ge7$ the inequalities
$\frac{r_n}6>k>\frac{rA}{a_n}~(=\frac6{a_2})$ are satisfied by $k=1$,
so we are in a case which we had already dealt with, and the proof is
complete.\qed

\end{document}